\documentclass[a4paper,11pt,reqno]{amsart}
\usepackage{amsfonts}
\usepackage{amsmath,amsthm}
\usepackage{enumerate}
\usepackage{hyperref}
\usepackage[T1]{fontenc}
\usepackage[utf8]{inputenc}
\usepackage{mathtools}
\usepackage{mathrsfs}
\usepackage{amssymb} 
\usepackage[text={6.0in,8.6in},centering]{geometry}
\usepackage{bigints}

\usepackage{color}
\usepackage{url}

\newtheorem{theorem}{Theorem}
\newtheorem{lemma}{Lemma}

\newtheorem{remark}{Remark}

\newcommand{\bM}{{\overline{M}}}

\begin{document}
\title{Concentration on submanifolds of positively curved homogeneous spaces}
\author{Nicolò De Ponti}

\address{Nicol\`o De Ponti: Dipartimento di Matematica ``Felice Casorati'', Universit\`a degli Studi di Pavia, Via Ferrata 5, 27100 Pavia (Italy)}

\email{nicolo.deponti01@universitadipavia.it}

\begin{abstract}
A classical result of Milman roughly states that every Lipschitz function on $\mathbb{S}^n$ is almost constant on a sufficiently high-dimensional sphere $\mathbb{S}^m\subset \mathbb{S}^n$. In this paper we extend the result by proving that any Lipschitz function on a positively curved homogeneous space is almost constant on a high dimensional submanifold.
\end{abstract}

\maketitle

\section*{Introduction}
The celebrated Lévy's concentration of measure inequality for the sphere, together with the work of V. Milman on the asymptotic behavior of Banach spaces, put forward the concentration of measure phenomenon in high dimensional spaces. Among the several results achieved in this field, let us mention the work on Banach spaces \cite{Milman, Ledoux}, infinite-dimensional groups \cite{Pestov1}, Riemannian manifolds \cite{Gromov2, Ledoux} or even general metric measure spaces \cite{Gromov, Burago}. In \cite{Milman2, Milman3}, Milman extends the idea of concentration to some homogeneous structure like Stiefel and Grassmann manifolds of an infinite dimensional Hilbert space. Some very clever applications of these results are proved by Gromov and Milman in \cite{Gromov1} and \cite{Milman4}. We also point out that recently Faifman, Klartag, and Milman \cite{Faifman} have discovered that a similar result also holds on the torus, where the strong concentration property is not available due to the flatness of the space.

The fundamental idea underlying these results is that a Lipschitz function tends to asymptotically concentrate near a single value. This type of results is usually stated in the sense of the measure, meaning that the probability of the subset where the function is almost constant tends to $1$ when the dimension of the space approaches infinity. Nevertheless, especially from the geometric point of view, it is important to find more structured subsets on which the function is concentrated. A well-known result in this direction, due to Milman \cite{Milman1}, roughly states that every Lipschitz function on $\mathbb{S}^n$ is almost constant on a sufficiently high-dimensional sphere $\mathbb{S}^m\subset \mathbb{S}^n$.  The aim of this note is to extend the result of Milman to the class of positively curved homogeneous space. Our main result, Theorem \ref{th: main}, will be stated in section \ref{sec: main result}, after all the necessary preliminaries. In the last section of the paper we also provide some explicit examples of spaces where the result can be applied.

\section{Preliminaries}
We now briefly recall some elementary facts of Riemannian geometry, mainly to fix the notation. In the paper we will use the overline notation for a quantity defined on the ambient manifold, while the same quantity intrinsically defined on a submanifold will not have the bar. 

We consider a smooth, connected, complete, $n$-dimensional Riemannian manifold $(\overline{M},\overline{g})$. We denote by $T_x\bM$ the tangent space at the point $x\in \bM$. Let $P\subset T_x\bM$ be a $2$-plane spanned by $(v,w)$, the sectional curvature at $x$ is defined by 
\begin{equation*}
\overline{Sec}_x(P)=\frac{\overline{R}_x(v,w,v,w)}{\overline{g}(v,v)\overline{g}(w,w)-(\overline{g}(v,w))^2},
\end{equation*}
where $\overline{R}_x$ is the Riemannian $(0,4)$ curvature tensor at the point $x$. 
The Ricci curvature at the point $x\in \bM$ in the direction $v\in T_x\bM$ is defined as
\begin{equation*}
\overline{Ric}_x(v)=\sum_{j=2}^n \overline{Sec}_x(P_{j})
\end{equation*}
where $\{v,e_2,...,e_n\}$ is a basis of $T_x\bM$ and $P_{j}$ denotes the $2$-plane spanned by $(v,e_j)$. 

We write $\overline{Sec}(\overline{M})\geq K$ (resp. $\overline{Ric}(\bM)\ge K$) standing for $\overline{Sec}_x(P)\geq K$ for every $x\in \overline{M}$ and every $2$-plane $P\subset T_x\overline{M}$ (resp. $\overline{Ric}_x(v)\ge K$ for every $x\in \overline{M}$ and every $v\in T_x\overline{M}$). 
In particular, if $\overline{Sec}(\overline{M})\geq K$ then $\overline{Ric}\ge (n-1)K$.

We recall that $\overline{Sec}(\overline{M})\geq K>0$ implies that $\overline{M}$ is compact as a consequence of the classical Myers's theorem (\cite{Myers}).

\subsection{Riemannian manifold as a metric measure space and the standard concentration theorem}
Let $\gamma:[a,b]\rightarrow \overline{M}$ be a smooth curve in $\overline{M}$, we define the length of $\gamma$ as $$L(\gamma)=\int_a^b\sqrt{\overline{g}\left(\dot{\gamma}(t),\dot{\gamma}(t)\right)}dt.$$
We endow $(\overline{M},\overline{g})$ with the Riemannian distance 
\begin{equation*}
d_{\overline{g}}(x_0,x_1):=\inf\{L(\gamma)\,|\, \gamma:[a,b]\rightarrow\overline{M} \ \textrm{smooth}, \, \gamma(a)=x_0, \, \gamma(b)=x_1\}.
\end{equation*}
It is well known that the Riemannian distance induces the same topology of the manifold (see e.g. \cite[Proposition 2.91]{Gallot}).

Furthermore, every Riemannian manifold has a Riemannian measure $\overline{\nu}$. It is a Borel measure and, if $\overline{M}$ is compact, it is finite, i.e $V:=\overline{\nu}(\overline{M})<\infty$. In the paper we consider on $\bM$ the measure $\overline{\mu}:=\overline{\nu}/V\in \mathscr{P}(\bM)$, the so-called \emph{normalized} Riemannian measure. Here $\mathscr{P}(\bM)$ denotes the set of probability measures over $\bM$.

A median for a measurable function $T:(\bM,\overline{\mu})\rightarrow \mathbb{R}$ is a number $m_T\in \mathbb{R}$ such that $\overline{\mu}(T\leq m_T)\geq \frac{1}{2}$ and $\overline{\mu}(T\geq m_T)\geq \frac{1}{2}$.

We now recall the classical concentration theorem for Riemannian manifolds with positive Ricci curvature, which follows from the Levy-Gromov isoperimetric inequality (see \cite[Theorems 2.3 and 2.4]{Ledoux}):

\begin{theorem}
\label{th:Levy-Gromov}
Let $(\overline{M},\overline{g})$ be an $n$-dimensional Riemannian manifold equipped with the Riemannian distance $d_{\overline{g}}$ and the normalized Riemannian measure $\overline{\mu}$. Suppose that $\overline{Ric}(\overline{M})\geq K(n-1)>0$. Then, for every 1-Lipschitz function $T:\bM\rightarrow \mathbb{R}$ and for every $\varepsilon>0$ 
\begin{equation}
\overline{\mu}(\{|T-m_T|\le \varepsilon\})\geq 1-\sqrt{\pi/2}\exp\left(-K(n-1)\varepsilon^2/2\right),
\end{equation}
where $m_T\in \mathbb{R}$ is a median for $T$.
\end{theorem}

\subsection{Isometry group, homogeneous spaces and Haar measure.}\label{sec: hom}
We denote with $Iso(\overline{M})$ the isometry group of $\overline{M}$, i.e. the set of maps $f:\overline{M}\rightarrow \overline{M}$ such that $\overline{g}(X,Y)=\overline{g}(f_{*}X,f_{*}Y)$ for every vector fields $X,Y$, together with the operation of composition. Here $f_{*}$ denotes the pushforward by $f$.

Myers and Steenrod \cite{Myers} proved that every isometry is a metric isometry on $(\bM,d_{\overline{g}})$ and $Iso(\overline{M})$ is a Lie group. 

The group $Iso(\overline{M})$ acts naturally on $\overline{M}$ via the map $f\cdot x=f(x)$.

We say that $\overline{M}$ is an \emph{homogeneous space} if this action is transitive, i.e. for every $x,y\in \bM$ there exists $f\in Iso(\bM)$ such that $f(x)=y.$ 

When $\overline{M}$ is compact, $Iso(\overline{M})$ is a compact Lie group and we can equip it with the unique left invariant Haar probability measure $\theta$ (see \cite[Theorem 1.129]{Gallot}). The measure $\theta$ induces a measure on $\overline{M}$ in the following way: let $x\in \bM$ and consider the map 
$$h^x:Iso(\overline{M})\rightarrow \overline{M}, \qquad h^x(f):=f(x).$$
We define $\overline{\mu}^x \in \mathscr{P}(\bM)$ as the push-forward of the measure $\theta$ through the map $h^x$, i.e. 
\begin{equation*}
\overline{\mu}^x(A):=\theta\left((h^x)^{-1}(A)\right) \  \textrm{for every Borel set} \ A\subset \bM.
\end{equation*}
The measures $\{\overline{\mu}^x : x\in \overline{M}\}$ and the normalized Riemannian measure $\overline{\mu}$ actually coincide, at least when $\bM$ is a compact homogeneous space (see \cite[Theorem 1.3]{Milman}).  

To sum up the previous discussion, we state the following lemma:
\begin{lemma}
\label{lemma1}
Let $\overline{M}$ be a compact homogeneous space and let $\overline{\mu}$ be the normalized Riemannian measure. Then, for every $x\in \overline{M}$ and for every Borel set $A\subset \bM$ we have 
\begin{equation}
\overline{\mu}(A)=\theta\left(\{f \in Iso(\overline{M}): f(x)\in A\}\right),
\end{equation}
where $\theta$ is the unique Haar left invariant probability measure on the compact Lie group $Iso(\overline{M})$.
\end{lemma}

\subsection{Totally geodesic submanifolds}
Let $(M,g)$ be a complete Riemannian submanifold of $(\overline{M},\overline{g})$, i.e. a submanifold $M\subset \bM$ endowed with the first fundamental form $g:=\iota^{*}\overline{g}$, where $\iota^{*}$ denotes the pullback by the inclusion map $\iota:M\rightarrow \bM$. 

We say that $M$ is a \emph{totally geodesic} submanifold if every geodesic on $(M,g)$ is also a geodesic on $(\overline{M},\overline{g})$. 

We write $\nabla$ and $\overline{\nabla}$ for the Levi-Civita connections on $M$ and $\overline{M}$, respectively. We recall that the second fundamental form is the symmetric tensor field defined by 
\begin{equation*}
\mathrm{I\!I}(X,Y)=\overline{\nabla}_XY-\nabla_XY,
\end{equation*}
where $X,Y$ are vector fields on the submanifold $M$.

A classical fact is that $M$ is totally geodesic if and only if the second fundamental form $\mathrm{I\!I}$ vanishes, i.e 
\begin{equation}
\label{curvatura}
\overline{\nabla}_XY=\nabla_XY \ \textrm{for every vector fields} \ X,Y \ \textrm{on} \ M.
\end{equation}

A direct consequence of \eqref{curvatura} is that the Riemannian curvature tensors $\overline{R}$ of $\bM$ and $R$ on $M$ agree on the domain of $R$. In particular, it follows that 
\begin{equation}\label{Sec implies Ric}
\overline{Sec}(\bM)\ge K \quad \textrm{implies} \quad Ric(M)\ge (m-1)K
\end{equation}
for every $m$-dimensional totally geodesic submanifold $M$.

\section{Main result}\label{sec: main result}
We are now ready to state our main result:
\begin{theorem}
\label{th: main}
Let $(\bM,\overline{g})$ an $n$-dimensional, homogeneous space endowed with the geodesic distance $d_{\overline{g}}$ and the normalized Riemannian measure $\overline{\mu}$, $n\ge 2$. Let us suppose that $\overline{Sec}(\overline{M})\geq K>0$. Then, for every $1$-Lipschitz function $T:\overline{M}\rightarrow \mathbb{R}$ and for every $\varepsilon$ with $2\pi/\sqrt{K}>\varepsilon>0$, there exists an $s$-dimensional submanifold $S\subset \overline{M}$ such that $|T(x)-m_T|<\varepsilon$ for every $x\in S$, where $m_T$ is a median of $T$ and $s$ is the largest dimension of a totally geodesic submanifold contained in $\bM$ satisfying
\begin{equation}\label{cond on s}
s< \left(\frac{\varepsilon^2K(n-1)}{8}-\ln{\sqrt{\pi/2}}\right)\frac{1}{\ln\left(\frac{2\pi}{\varepsilon\sqrt{k}}\right)}
\end{equation}
\end{theorem}

\begin{remark}
By recalling that $\mathbb{S}^s$ can be seen as a totally geodesic submanifold of the sphere $\mathbb{S}^n$ for every $s\le n$, we recover the classical result of Milman (see \cite{Milman1} and \cite[Theorem 2.4]{Milman}) as a particular case.
\end{remark}

\begin{proof}
Let $M_0$ be an $s$-dimensional totally geodesic submanifold of $\overline{M}$ and let 
$$Y:=\{M=f(M_0) \, | \, f\in Iso(\overline{M})\}.$$ 
The set $Y=Iso(\overline{M})/H$ can be endowed with a structure of manifold, where $H$ is the closed Lie subgroup of all $f\in Iso(\overline{M})$ such that $f(M_0)=M_0$. Let $dy$ the  $Iso(\overline{M})$-invariant probability measure on $Y$. Then for every continuous function $u:\overline{M}\rightarrow \mathbb{R}$ it follows
\begin{equation}
\label{disintegration}
\int_{\overline{M}}u(x)d\overline{\mu}(x)=\int_Y\left(\int_M u(x)d\mu_M(x)\right) dy,
\end{equation}
where $\overline{\mu}$ is the normalized Riemannian measure on $\overline{M}$ and $\mu_M$ is the normalized Riemannian measure on $M$ (see \cite[Chapter 1]{Helga}). 

Let $I_{A^{\varepsilon}}$ be the characteristic function of the set $A^{\varepsilon}:=\{|T(x)-m_T|\leq \varepsilon/2 \}$, where $2\pi/\sqrt{K}>\varepsilon>0$. By using now the identity $\eqref{disintegration}$ with a monotone sequence of continuous functions converging to  $I_{A^{\varepsilon}}$ (which exists since $A^{\varepsilon}$ is a closed subset of a metric space) and by applying Theorem \ref{th:Levy-Gromov}, we can infer the existence of a totally geodesic submanifold $(S,g)$ such that 
\begin{equation}\label{prima condizione misura}
\mu_S(S\cap A^{\varepsilon})\geq 1-\sqrt{\pi/2}\exp(-\varepsilon^2K(n-1)/8).
\end{equation}

Using the implication \eqref{Sec implies Ric} we know that $S$ has Ricci curvature bounded below by $(s-1)K>0$, so that we can apply the Bishop-Gromov Theorem \cite[Theorem 4.19]{Gallot} and for any $x\in S$ we have
\begin{equation}\label{lem: long ineq}
\mu_S(B_{x}(\varepsilon/2))\ge \frac{Vol(B_K(\varepsilon/2))}{Vol(B_K(D))},
\end{equation}
where $B_x(r):=\{y\in S : d_g(x,y)<r\}$, $\mu_S$ is the normalized Riemannian measure on $S$, $D$ is the diameter of $S$ and $Vol(B_K(D))$ denotes the volume of the ball of radius $D$ in the $s$-dimensional sphere of constant sectional curvature $K$.
By recalling that $D\leq \pi/\sqrt{K}$ as a consequence of Myers Theorem \cite{Myers}, we obtain that for any $x\in S$ 
\begin{equation}\label{seconda condizione misura}
\mu_S\left(B_x\big(\varepsilon/2\big)\right)\ge \frac{\int_0^{\varepsilon/2}\left(\sin(\sqrt{K}t)\right)^{s-1}dt}{\int_0^{\pi/ \sqrt{K}}\left(\sin(\sqrt{K}t)\right)^{s-1}dt}=\frac{\int_0^{\varepsilon\sqrt{K}/2}\left(\sin{t}\right)^{s-1}dt}{\int_0^{\pi}(\sin{t})^{s-1}dt}\ge \bigg(\frac{\varepsilon\sqrt{K}}{2\pi}\bigg)^s
\end{equation} 
where the last inequality follows since 
$$g(x):=\frac{\int_0^x \left(\sin{t}\right)^{s-1}dt}{x^s}$$
is decreasing in $[0,\pi]$ for every $s\ge 1$. Indeed,
$$g'(x)=\frac{x^s (\sin{x})^{s-1}-sx^{s-1}\int_0^x \left(\sin{t}\right)^{s-1}dt}{x^{2s}}$$
and 
$$h(x):=x(\sin{x})^{s-1}\leq s\int_0^x \left(\sin{t}\right)^{s-1}dt:=w(x)$$
since $h(0)=w(0)$ and 
$$h'(x)=(\sin{x})^{s-1}+(s-1)x(\sin{x})^{s-2}\cos{x}\le s(\sin{x})^{s-1}=w'(x)$$
as a consequence of the inequality
$$x\cos(x)\le \sin(x) \qquad x\in[0,\pi].$$

If
\begin{equation}\label{condizione intersezione}
\mu_S(S\cap A^{\varepsilon})+\mu_S\left(B_x(\varepsilon/2) \right)>1,
\end{equation}
we obtain that any ball $B_x(\varepsilon/2)$ in the $s$-dimensional submanifold $S$ intersects $A^{\varepsilon}$. In particular, using \eqref{prima condizione misura} and \eqref{seconda condizione misura}, the inequality \eqref{condizione intersezione} is certainly satisfied if 
\begin{equation}
1-\sqrt{\pi/2}\exp(-\varepsilon^2K(n-1)/8)+\bigg(\frac{\varepsilon\sqrt{K}}{2\pi}\bigg)^s>1,
\end{equation}
i.e. if $s$ satisfies the bound stated in \eqref{cond on s}.

In this situation, for any $x\in S$ there exists $z\in B_x(\varepsilon/2)\cap A^{\varepsilon}$, so that
\begin{equation*}
\begin{aligned}
&T(x)=T(x)-T(z)+T(z)\leq |T(x)-T(z)|+T(z)< \varepsilon/2 + m_T+\varepsilon/2=m_T+\varepsilon \, ,\\
&T(x)=T(x)-T(z)+T(z)\geq -|T(x)-T(z)|+T(z)> -\varepsilon/2 + m_T-\varepsilon/2=m_T-\varepsilon \, ,
\end{aligned}
\end{equation*}
where we have used the Lipschitz condition on $T$. In particular $|T(x)-m_T|<\varepsilon$ for every $x\in S$.
\end{proof}

\begin{remark}
We notice that, by taking a slightly smaller $s$, Theorem \ref{th: main} is still true for two different Lipschitz functions on the same submanifold (see \cite[Remark 2.9]{Milman}).
\end{remark}

\section{Symmetric spaces and examples}
An isometry $f:\overline{M}\rightarrow \overline{M}$ is called involutive if $f\circ f=Id$, the identity isometry. 

$\overline{M}$ is a symmetric space if, for each point $x\in \overline{M}$, there exists an involutive isometry $f_x$ such that $x$ is an isolated fixed point of $f_x$. 

A symmetric space is an homogeneous space \cite[pag. 223]{Kobayashi}.

A lot is known about totally geodesic submanifolds in symmetric spaces (see \cite{ChenSurvey} for an excellent exposition). Here it is useful to recall the following two facts:

\begin{enumerate}[1)]

\item A complete totally geodesic submanifold of a symmetric space is a symmetric space.
\item Let $\overline{M}$ be an $n$-dimensional symmetric space, $n\geq 2$, then there exists a complete totally geodesic submanifold $M$ whose dimension satisfies $n/2 \leq dim(M) <n.$
\end{enumerate}

The first assertion is a standard result, which can be found in \cite{Kobayashi}. The second statement was proved by Chen and Nagano in \cite{Chen}. Thus, the class of symmetric spaces provides an example of homogeneous manifolds that possesses sufficiently high dimensional totally geodesic submanifolds. 

To give some explicit examples where Theorem \ref{th: main} can be applied, besides the aforementioned case of the sphere, we recall that the real projective space $\mathbb{RP}^n$ is a $n$-dimensional symmetric space of constant sectional curvature equal to $1$ and $\mathbb{RP}^{n-1}\subset \mathbb{RP}^n$ is a totally geodesic submanifold.

The complex projective space $\mathbb{CP}^n$ is a $2n$-dimensional manifold. We endowed it with the Fubini–Study metric, so that it is a symmetric space with sectional curvature $Sec(\mathbb{CP}^n)\geq \frac{1}{4}$. $\mathbb{RP}^n$ and $\mathbb{CP}^{n-1}$ are the maximal totally geodesic submanifolds of $\mathbb{CP}^n$.

The quaternionic projective space $\mathbb{HP}^n$ is a $4n$-dimensional manifold. Equipped with the Fubini–Study metric, it is a symmetric space with sectional curvature $Sec(\mathbb{HP}^n)\geq \frac{1}{4}$ and it possesses $\mathbb{CP}^{n}$ and $\mathbb{HP}^{n-1}$ as maximal totally geodesic submanifolds.

\section{Acknowledgements}

The author would like to thank Vitali Milman for the precious discussion that led to an improvement of the introduction and the bibliography.
The author is also thankful to Emanuele Casini and Stefano Pigola for reading an earlier version of the paper and suggesting some improvements. This work is partially supported by GNAMPA–INDAM Project 2019 (Italy) ``Trasporto Ottimo per Dinamiche con Interazione''.


\begin{thebibliography}{}
\bibitem{Burago} Burago, D., Burago, S., Ivanov, S., \emph{A course in metric geometry}, American Matematical Society, (2001).
\bibitem{Chen}  Chen, B.Y., Nagano, T., \emph{Totally geodesic submanifolds of symmetric spaces. II.}, Duke Math. J. 45, no. 2, 405–425, (1978).
\bibitem{ChenSurvey} Chen, B.Y., \emph{Riemannian Submanifolds}, Handbook of differential geometry, Dillen, F., Verstraelen, L., Elsevier Science, 187-418, (2000).
\bibitem{Faifman} Faifman, D., Klartag, B., Milman V., \emph{On the oscillation rigidity of a Lipschitz function on a high-dimensional flat torus}, Geometric aspects of functional analysis, 123–131, (2014).
\bibitem{Gallot} Gallot, S., Hulin, D., Lafontaine, J., \emph{Riemannian Geometry}, Springer, (2004). 
\bibitem{Gromov} Gromov, M., \emph{Metric structure for Riemannian and non-Riemannian space}, Birkhauser, (1999).
\bibitem{Gromov1} Gromov, M., Milman, V. D., \emph{A topological application of the isoperimetric inequality}, Amer. J. Math. 105, no. 4, 843–854, (1983).
\bibitem{Gromov2} Gromov, M., \emph{Filling Riemannian manifolds}, J. Differential Geom. 18, no. 1, 1–147,(1983).
\bibitem{Helga} Helgason, S,, \emph{Groups and geometric analysis}, Academic Press, (1984).
\bibitem{Kobayashi} Kobayashi, S., Katsumi, N., \emph{Foundations of differential geometry, Volume 2}, Interscience, (1969).
\bibitem {Ledoux} Ledoux, M., \emph{The concentration of measure phenomenon}, American Matematical Society (2001).
\bibitem{Myers} Myers, S. B., Steenrod, N.E., \emph{The Group of Isometries of a Riemannian Manifold}, Annals of Mathematics, Second Series, Vol. 40, No. 2 400-416, (1939).
\bibitem{Milman} Milman, V. D., Schechtman, G., \emph{Asymptotic theory of finite dimensional normed space}, Springer, (1986).
\bibitem{Milman1} Milman, V. D.,\emph{A new proof of A. Dvoretzky's theorem on cross-sections of convex bodies}, (Russian), Funkcional. Anal. i Priložen. 5 , no. 4, 28–37, (1971).
\bibitem{Milman2} Milman, V. D. \emph{Asymptotic properties of functions of several variables that are defined on homogeneous spaces}, (Russian), Soviet Math. Dokl. 12 (1971), 1277–1281; translated from Dokl. Akad. Nauk SSSR 199, 1247--1250, (1971).
\bibitem{Milman3} Milman, V. D., \emph{A certain property of functions defined on infinite-dimensional manifolds}, (Russian) Dokl. Akad. Nauk SSSR 200, 781–784, (1971).
\bibitem{Milman4} Milman, V. D.,\emph{Diameter of a minimal invariant subset of equivariant Lipschitz actions on compact subsets of $R^k$}, Geometrical aspects of functional analysis (1985/86), 13–20, Lecture Notes in Math., 1267, Springer, Berlin, (1987).
\bibitem{Milnor} Milnor, J., \emph{Curvature of Left Invariant Metrics on Lie Group},  Advances in Mathematics 21, 293-329, (1976).
\bibitem{Pestov} Pestov, V., \emph{Ramsey-Milman phenomenon, Urysohn metric spaces, and extremely amenable groups}, Israel J. Math. 127, 317–357, (2002).
\bibitem{Pestov1} Pestov, V., \emph{Dynamics of infinite-dimensional groups. The Ramsey-Dvoretzky-Milman phenomenon}, University Lecture Series, 40. American Mathematical Society, Providence, RI, viii+192 pp., (2006).
\end{thebibliography}
\end{document}